\documentclass[11pt]{amsart}
\usepackage{color,array,amssymb,amsmath, amsthm, amsfonts, tikz-cd}

\usepackage{graphicx}
\usepackage{listings}
\usepackage[margin=1in]{geometry}
\usepackage{lstautogobble}
\usepackage{enumerate}
\usepackage{thmtools}
\usepackage{thm-restate}
\usepackage{amsthm}
\usepackage{verbatim}

\usepackage{mathtools}
\usepackage{physics}
\usepackage{color}
\usepackage{hyperref}
\usepackage[capitalise]{cleveref}
\usepackage[bottom]{footmisc}
\crefformat{equation}{(#2#1#3)}

\usepackage{float}
\restylefloat{table}

\usepackage{mathrsfs}

\theoremstyle{plain}
\newtheorem{thm}{Theorem}[section]
\newtheorem{prop}[thm]{Proposition}
\newtheorem{lemma}[thm]{Lemma}

\theoremstyle{definition}
\newtheorem{mydef}[thm]{Definition}

\newtheorem{remark}[thm]{Remark}

\Crefname{prop}{Proposition}{Propositions}

\numberwithin{equation}{section} 

\DeclarePairedDelimiter{\paren}{\lparen}{\rparen}

\newcommand{\M}{{\mathcal{M}}}

\renewcommand{\d}{\mathsf{d}}

\newcommand{\R}{{\mathbb{R}}}

\newcommand{\T}{{\mathbb{T}}}
\newcommand{\g}{{\mathsf{g}}}

\renewcommand{\M}{{\mathbb{M}}}

\newcommand{\tl}{\tilde}

\newcommand{\D}{\Delta}

\newcommand{\ph}{\phantom{=}}
\newcommand{\nn}{\nonumber}

\newcommand{\ux}{\underline{x}}

\newcommand{\ep}{\varepsilon}

\newcommand{\be}{\beta}
\newcommand{\ka}{\kappa}

\newcommand{\XN}{X_N}
\newcommand{\YN}{Y_N}

\newcommand{\Cs}{\mathsf{C}}
\newcommand{\Fr}{{F}}
\newcommand{\Ec}{\mathcal{E}}
\renewcommand{\P}{\mathcal{P}}

\newcommand{\PNbeta}{\mathbb{P}_{N,\beta}}
\newcommand{\Esp}{\mathbb{E}}

\let\div\relax
\DeclareMathOperator{\div}{div}

\def\XXint#1#2#3{{\setbox0=\hbox{$#1{#2#3}{\int}$ }
\vcenter{\hbox{$#2#3$ }}\kern-.6\wd0}}

\def \hal{\frac{1}{2}}
\def\({\left(}
\def\){\right)}
\def \ep{\varepsilon}
\def\nab{\nabla}
\def\indic{\mathbf{1}}

\title[Modulated LSI and generation of chaos]{Modulated logarithmic Sobolev inequalities and generation of chaos}
\author[M. Rosenzweig]{Matthew Rosenzweig}
\address{Matthew Rosenzweig, Carnegie Mellon University, Department of Mathematical Sciences, Pittsburgh, PA} 
\email{mrosenz2@andrew.cmu.edu}
\thanks{M. R. is supported in part by the Simons Foundation through the Simons Collaboration on Wave Turbulence and by NSF grants DMS-2052651, DMS-2206085.}
\author[S. Serfaty]{Sylvia Serfaty}
\address{Sylvia Serfaty, Courant Institute of Mathematical Sciences, New York University, New York City, NY}
\email{serfaty@cims.nyu.edu}
\thanks{S. S. is supported  by the Simons Foundation through a Simons Investigator award, by NSF grant DMS-2247846, and by an excellence chair of the Fondation Sciences Math\'ematiques de Paris and PSL}

\begin{document}
\begin{abstract}
We consider mean-field limits for  overdamped Langevin dynamics of $N$ particles with possibly singular interactions. 
It has been shown that a modulated free energy method can be used to prove the mean-field convergence or propagation of chaos for a certain class of interactions, including Riesz kernels. We show here that generation of chaos, i.e. exponential in time convergence  to a tensorized (or iid) state starting from a nontensorized one, can be deduced from the modulated free energy method provided  a uniform-in-$N$ ``modulated logarithmic Sobolev inequality" holds. Proving such an inequality is  a question of independent interest, which is generally difficult. 
As an illustration, we show  that uniform  modulated logarithmic Sobolev inequalities can be proven for a class of situations in one dimension.
\end{abstract}
\maketitle

\section{Introduction}
Consider a canonical Gibbs measure for $N$ particles with energy $\mathcal H_N(\XN)$, with $X_N \coloneqq (x_1, \dots, x_N) $, $x_i\in \R^\d$, of the form 
\begin{equation}
\label{gibbs}
d\mathbb{P}_{N,\beta}(\XN)= \frac{1}{Z_{N,\beta}} e^{-\beta \mathcal H_N(\XN)} dx_1\dots x_N.\end{equation}
It is well know that if $\mathbb{P}_{N,\beta}$ satisfies a  Poincar\'e  (or spectral gap) or logarithmic Sobolev inequality (LSI), with a constant independent of $N$, then the joint law of the particles under the overdamped Langevin (Glauber) dynamics 
\begin{equation}
\label{langevin}
dx_{i}^t = -\nabla_i \mathcal H_N(\XN) +  \sqrt{\frac{2}{\beta}}dW_i^t \qquad i\in\{1,\ldots,N\},
\end{equation}
where the $W_i^t$ are independent standard Brownian motions, converges exponentially fast in time to the steady state $\PNbeta$. See, for instance, \cite{Villani2004,BGL2014}. A Poincar\'e inequality or LSI is satisfied as soon as $\mathcal{H}_N$ satisfies a uniform strict convexity condition of the form $\mathrm{Hess} \ \mathcal{H}_N \ge c I_{dN\times dN}$ with $c>0$, with the Poincar\'{e}/LSI constant only depending on $c$ \cite{BE1985}. Proving uniform LSIs  meaningfully beyond this uniformly convex case is in general hard and the object of current efforts. We refer to \cite{BB2019spin,BB2021sg} for some instances of progress. For a taste of the extensive literature on LSIs, we refer to \cite{ABCFGIMRS2000}.

In this note, we are interested in the particular case of pair interaction energies of the form 
\begin{equation}
\label{energy}
\mathcal H_N(\XN)= \frac1{2N} \sum_{1\leq i\neq j\leq N} \g(x_i,x_j) + \sum_{i=1}^N V(x_i),
\end{equation}
where again $\XN= (x_1, \dots, x_N ) \in (\R^\d)^N$; $\g:(\R^\d)^2\rightarrow [-\infty,\infty]$ is some symmetric interaction potential belonging to a class to be specified later and is similar to that considered in \cite{JW2018,BJW2019edp,BJW2020,CdCRS2023},  which includes repulsive Coulomb and Riesz interactions of the form 
\begin{equation}\label{rieszcase}
\g(x,y) = \begin{cases} -\log|x-y|, & {s=0} \\ \frac{1}{s} |x-y|^{-s}, & {s<\d}, \end{cases}
\end{equation}
 as well as moderately attractive ones; and $V$ is some confinement potential.

In that case, the  overdamped Langevin  dynamics  is  of the form
\begin{equation}
\label{eq:SDE}
\begin{cases}
dx_{i}^t =  \displaystyle \( -\frac{1}{N}\sum_{1\leq j\leq N : j\neq i} \nabla_1\g(x_i^t,x_j^t)  -\nabla V(x_i^t)   \) dt + \sqrt{\frac{2}{\beta}}dW_i^t\\
x_i^t|_{t=0} = x_i^0
\end{cases}\qquad i\in\{1,\ldots,N\},
\end{equation}
with  $x_i^0\in\R^\d$ the pairwise distinct initial positions. Here, $\nabla_1$ denotes the gradient with respect to the first argument of $\g$. The mean-field limit, or equivalently {\it propagation of chaos}, for such evolutions has been proved for $\g$ sufficiently regular by many classical methods \cite{JW2017_survey},  for $\g$ possibly singular (with $V=0$) by the relative entropy method in \cite{JW2018}, and for $\g$ even more singular and Coulomb/Riesz-like  by the modulated free energy method \cite{BJW2019edp,BJW2020,CdCRS2023}  using the modulated energy  of \cite{Duerinckx2016,Serfaty2020}.  We will recall these methods below but suffice it to say that propagation of chaos means that if the initial data is distributed according to the probability  distribution $f_N^0(\XN) 
= \mu^0(x_1) \dots \mu^0(x_N)$ on $(\R^\d)^N$ (i.e., the particles are iid with common law $\mu^0$), then the solution  $f_N^t$ of the $N$-particle Liouville/forward Kolmogorov equation associated to the dynamics \eqref{eq:SDE} is such that 
\begin{equation}\label{eq:pc}
f_{N,k}^t\rightharpoonup (\mu^t )^{\otimes k} \quad \text{as}  \ N\to \infty,
\end{equation}
where $f_{N,k}^t$ denotes the $k$-point marginal of $f_N^t$ and $\mu^t$ is a solution to the mean-field evolution\footnote{In \eqref{meanfield} and the remainder of the paper, we abuse the convolution notation by defining $\g\ast\mu(x) \coloneqq \int_{\R^\d}\g(x,y)d\mu(y)=\int_{\R^\d}\g(y,x)d\mu(y)$, since $\g$ is assumed to be symmetric.}
\begin{equation}\label{meanfield}
\begin{cases} 
\partial_t \mu^t - \div (( \nab \g*\mu^t  + \nabla V) \mu^t)=\frac{1}{\beta} \Delta \mu^t\\
\mu^t|_{t=0}= \mu^0.
\end{cases}
\end{equation}
The convergence \eqref{eq:pc} is for fixed $k$, as $N \to \infty$. There has been recent progress on understanding the optimal rate of this convergence in the context of the relative entropy method \cite{Lacker2023}. The modulated free energy method yields convergence in relative entropy, which in turn implies convergence of all the fixed marginals. Here, the (normalized) relative entropy is defined by 
\begin{equation}\label{relent}
H_N(f_N|g_N) \coloneqq \frac{1}{N} \int_{(\R^\d)^N} \log\left(\frac{f_N}{g_N}\right) df_N.
\end{equation}
There has also been progress on showing bounds for the relative entropy which vanish as $N\rightarrow\infty$ and hold uniformly in time, hence proving uniform-in-time propagation of chaos in  \cite{RS2021,GlBM2021, GlBM2023,CdCRS2023,LlF2023}. Informally, the distance between the laws $f_N^t, (\mu^t)^{\otimes N}$ does not grow arbitrarily large as time becomes large.

The notion of {\it generation of chaos}, a term coined recently by Lukkarinen \cite{Lukkarinen2023}, consists in a similar convergence as time gets large, even when the initial data $f_N^0$ is not tensorized, i.e. does not exhibit chaos or independence. Interpreted in an entropic sense (see \cite{HM2014} for a discussion of various notions of chaos), we have  generation of chaos if $H_N(f_N^t |(\mu^t) ^{\otimes N}) \to 0 $ as $t \to \infty$, uniformly in $N$, and without any smallness assumptions on $H_N(f_N^0|(\mu^0)^{\otimes N})$. This is what we wish to demonstrate here holds, under  a {\it uniform-in-N modulated LSI} condition, that we will define below. 

\subsection{Modulated energies and modulated Gibbs measures}\label{sec1.1}
Before going further, let us review the notion of modulated energy. This object was first introduced as a next-order electric energy in \cite{SS2015,RS2016,PS2017} and used in the dynamics context as a modulated energy in \cite{Duerinckx2016,Serfaty2020} and following works---in the spirit of \cite{Brenier2000}.
Given a probability density $\mu$ on $\R^\d$,
we define the modulated energy of the configuration $\XN$ as 
\begin{equation}\label{defF}
F_N(\XN, \mu) \coloneqq \hal \int_{(\R^{\d})^2\setminus\triangle} \g(x,y) d\( \frac1N\sum_{i=1}^N \delta_{x_i}- \mu\) (x) 
d\( \frac1N\sum_{i=1}^N \delta_{x_i}- \mu\) (y) ,
\end{equation}
where $\triangle$ denotes the diagonal in $(\R^\d)^2$. This is the total interaction of the system of $N$ discrete charges at $x_i$ against a negative (neutralizing) background charge $\mu$, with the self-interaction of the points, which is infinite if $\g(x,x)=\infty$,\footnote{If $\g(x,x)$ is finite, then the renormalization is unnecessary. See \cref{ssec:introApps} for elaboration.} removed. As shown in the aforementioned prior works, $F_N$ is not necessarily positive; however,  under appropriate assumptions on $\g$, it acts in effect as a squared distance between the empirical measure $\frac1N\sum_{i=1}^N \delta_{x_i}$ and $\mu$.
Given the density $\mu$, we can also define the {\it modulated Gibbs measure }
\begin{equation}\label{defQ}
\mathbb{Q}_{N,\beta}(\mu) \coloneqq \frac{1}{K_{N,\beta}(\mu)} e^{-\beta  N F_N(\XN, \mu)} d\mu(x_1) \dots d\mu(x_N),
\end{equation}
where 
\begin{equation}\label{defKNbe}
K_{N,\beta}(\mu) \coloneqq \int_{(\R^\d)^N} e^{-\beta N F_N(\XN, \mu)} d\mu(x_1) \dots d\mu(x_N)
\end{equation}
is the associated partition function. An example of use of such a modulated Gibbs measure is provided in \cite{AS2021} in the study of \eqref{gibbs} for the energy \eqref{energy} in the case where $\g$ is the Coulomb interaction.

Following for instance \cite{AS2021,AS2022}, we may introduce in the context of \eqref{energy} the {\it thermal equilibrium measure} $\mu_\beta$, which is defined as the minimizer among probability densities of the mean-field  free energy 
\begin{equation}\label{MFenergy}
\mathcal E_\beta(\mu) \coloneqq \hal \int_{(\R^\d)^2} \g(x,y) d\mu(x)d\mu(y)+ \int_{\R^\d} V(x) d\mu(x) + \frac{1}{\beta}\int_{\R^\d}\log \mu(x)d\mu(x).
\end{equation}
If $V$ grows sufficiently fast at infinity, then $\mathcal E_\beta$ has a unique minimizer, which is characterized by the existence of a constant $c_\beta\in\R$ such that
\begin{equation}
\label{charactem}
\g * \mu_\beta + V + \frac{1}{\beta} \log \mu_\beta=c_\beta \quad \text{in}  \ \R^\d.
\end{equation}
In the Coulomb case, \cite{AS2022} studied how $\mu_\beta$ converges to the usual equilibrium measure as $\beta \to \infty$.

The thermal equilibrium measure allows, as seen in \cite{AS2021}, for a nice {\it splitting of the energy} and thus of the Gibbs measure, as follows. For any $\XN$, using \eqref{defF}, \eqref{charactem}, and direct computations, we have 
\begin{equation}
\mathcal H_N(\XN)= N  \mathcal E_\beta(\mu_\beta) +  N F_N(\XN, \mu_\beta) - \frac{1}{\beta}\sum_{i=1}^N \log \mu_\beta(x_i). \end{equation}
Inserting this identity into \eqref{gibbs}, we find that 
\begin{equation}\label{gibbsrewrite}
d\PNbeta^V(\XN) = \frac{e^{-\beta N\mathcal E_\beta(\mu_\beta) } }{Z_{N,\beta}^V}  e^{-\beta N F_N(\XN, \mu_\beta)} d\mu_\beta(x_1)\dots d\mu_\beta(x_N).\end{equation}
In other words, comparing with \eqref{defQ}, we have found that 
\begin{equation}
\label{gibbsmod}
\PNbeta^V =  \mathbb{Q}_{N,\beta}(\mu_\beta) 
\end{equation}
and 
\begin{equation}
Z_{N,\beta}^V= K_{N,\beta}(\mu_\beta) e^{-\beta N \mathcal E_\beta(\mu_\beta)}.
\end{equation}
Thus, the Gibbs measure is itself a modulated Gibbs measure, relative to the thermal equilibrium measure. 

 Conversely, given a probability measure $\mu$, it is easy to see the modulated Gibbs measure $\mathbb{Q}_{N,\beta}(\mu)$ as a Gibbs measure  through a change of the confining potential. Following \eqref{charactem}, let
 \begin{equation}\label{defVmub}
 V_{\mu,\beta}\coloneqq - \g * \mu - \frac1{\beta} \log \mu.
 \end{equation}
Then retracing the steps of the splitting formula above, one has 
 \begin{equation}\label{QP}
 \mathbb{Q}_{N,\beta}(\mu)= \mathbb{P}_{N,\beta}^{V_{\mu, \beta}}.
 \end{equation}

With the rewriting \eqref{gibbsmod}, a crucial condition, appearing in all that follows,  is 
\begin{equation}\label{condK}
 |\log K_{N,\beta}(\mu) |=o(N) 
\end{equation}
 with a $o(N)$ uniform in $\beta\in[\hal \beta_0,2\beta_0]$, for some fixed $\beta_0$, which corresponds for instance to the ``large deviations estimates" in \cite{JW2018}.  We will call it a \emph{smallness of the free energy}.  This condition---and even a stronger quantitative one---can be  proven in the  Riesz cases \eqref{rieszcase} and for bounded continuous interactions. We give a short proof of this fact in the {appendix}. In the attractive log case, it is proven in \cite{BJW2020} and later streamlined in \cite{CdCRS2023a}.

In several cases of interest, including in particular \eqref{rieszcase} (cf. \cite{Serfaty2020,NRS2021}), $F_N$  is  positive up to a small additive constant 
and  controls a form of distance (e.g., a squared Sobolev norm). In such cases one may easily obtain a concentration estimate around $\mu$ as follows.\footnote{If this is true for  $-F_N$ instead of $F_N$, the same reasoning below applies, using $2\beta$ and $\beta$ instead of $\beta/2$ and $\beta$ in \eqref{expmom1}.} 
 By definition \eqref{defQ} of $\mathbb{Q}_{N,\beta}(\mu)$, we may rewrite the exponential moments of the modulated energy $F_N$ as 
 \begin{equation}\label{expmom1}
 \log \Esp_{\mathbb{Q}_{N,\beta}(\mu)}\left[ e^{\frac{\beta}{2} N F_N(\XN, \mu)} \right] =  \log \frac{K_{N,\beta/2}(\mu)}{K_{N,\beta}(\mu)}.\end{equation}
If  \eqref{condK} holds, we obtain the exponential moment control
 \begin{equation}
 \label{expmom3}\left| \log \Esp_{\mathbb{Q}_{N,\beta}(\mu)}  \left[ e^{\frac{\beta}{2}  N F_N(\XN, \mu) } \right]\right|\le o(N).
 \end{equation}
 Thus, using the almost positivity of  $F_N$ and the fact that it  controls a squared distance between the empirical measure and reference density,  this provides a concentration estimate around $\mu$ and implies a law of large numbers in the form 
\begin{align}
\mathbb{E}_{\mathbb{Q}_{N,\beta}(\mu)}\left[\left\|\frac{1}{N}\sum_{i=1}^N \delta_{x_i}-\mu\right\|^2\right] \rightarrow 0,
\end{align}
where $\|\cdot\|$ is a suitable norm. By standard arguments, this convergence also implies propagation of chaos for the statistical equilibrium $\mathbb{Q}_{N,\beta}(\mu)$ (see for instance \cite{RS2016,CD2021}): 
  \begin{align}\label{123}
\mathbb{Q}_{N,\beta}^{(k)}(\mu) \xrightharpoonup[]{} \mu^{\otimes k} \quad \text{as} \ N\to \infty,
\end{align}
where $\mathbb{Q}_{N,\beta}^{(k)}(\mu) $ denotes the $k$-point marginal of $\mathbb{Q}_{N,\beta}(\mu)$ and $k$ is fixed.

\subsection{Modulated free energy}
We may now define the modulated free energy, as introduced in \cite{BJW2019crm,BJW2019edp,BJW2020}. Given a reference probability density $\mu$ on $\R^\d$ as above and a probability density $f_N$ on $(\R^\d)^N$, the modulated free energy is defined by 
\begin{align}\label{eq:MFE}
E_N(f_N, \mu) \coloneqq \frac1\beta H_N(f_N\vert \mu^{\otimes N}) +\mathbb{E}_{f_N}\left[F_N(\XN,\mu)\right],
\end{align}
where $H_N$ is the relative entropy as in \eqref{relent} and $\mathbb{E}_{f_N}$ denotes the expectation with respect to the measure $f_N$, viewing the $\XN$ as a random variable. Let us  remark here that using the explicit form of \eqref{defQ}, the modulated free energy can be rewritten as 
\begin{equation}\label{rewritemodenergy}
E_N(f_N, \mu)=  \frac{1}{\beta} \( H_N(f_N| \mathbb{Q}_{N,\beta}(\mu))+ \frac{\log K_{N,\beta}(\mu)}{N}\).
\end{equation}
 In other words,  up to a constant related to the  smallness of free energy condition \eqref{condK},  the modulated free energy is another relative entropy. Note that this provides an easy proof of the 
 fact that $E_N(f_N, \mu)$ is essentially positive if the smallness of free energy condition
 \eqref{condK} holds. 
  Moreover, 
  controlling the relative entropy from $f_N$ to $\mathbb{Q}_{N,\beta}$ proves closeness of the particle density 
   to $\mathbb{Q}_{N,\beta}(\mu)$ and is, in reality, more precise than the mean-field limit and propagation of chaos provided by the control of $H_N(f_N \vert \mu^{\otimes N})$. As $t \to \infty$, the solution $\mu^t$ to \eqref{meanfield} converges to the thermal equilibrium measure $\mu_\beta$, and $\mathbb{Q}_{N,\beta}(\mu_\beta)$ is, as already noticed in \eqref{QP}, equal to $\mathbb{P}_{N,\beta}$, so we retrieve the fact, provided by usual LSI, that there is convergence in large time to $\mathbb{P}_{N,\beta}$, the invariant measure for the dynamics \eqref{eq:SDE}. See \cref{sec1.5} below for a further discussion on the advantages of $\mathbb{Q}_{N,\beta}(\mu)$ over $\mathbb{P}_{N,\beta}$.
   Finally, if one wishes to retrieve closeness of $f_N$ to $\mu^{\otimes N}$, one may either use a control of the 
  negative part of the modulated energy by the relative entropy, as ensured by condition \eqref{assmp2'} below, or use the concentration inequality via its consequence \eqref{123}.

\subsection{Evolution of modulated energy, Fisher information, and uniform  LSI}
The crucial computation of \cite{BJW2019crm, BJW2019edp, BJW2020} (performed on the torus, but the whole-space  with confining potential case is similar) is that when  differentiating in time $E_N(f_N^t, \mu^t)$, for $f_N^t$ solving the forward Kolmogorov equation and $\mu^t$ solving the mean-field evolution equation \eqref{meanfield}, a cancellation occurs, leading to
\begin{multline}\label{eq:MFEdiss}
\frac{d}{dt}E_N(f_N^t, \mu^t)  \leq  -\frac{1}{2} \int_{(\R^\d)^N}\int_{(\R^\d)^2\setminus\triangle} (u^t(x)-u^t(y))\cdot \nabla_1\g(x,y) d\left(\frac1N\sum_{i=1}^N\delta_{x_i} - \mu^t\right)^{\otimes 2}(x,y)d f_N^t \\
-\frac{1}{\beta^2 N}\int_{(\R^\d)^N}\sum_{i=1}^N \left|\nabla_i \log\paren*{\frac{f_N^t}{(\mu^t)^{\otimes N}}} +\frac{\beta}{N} \sum_{j\neq i}  \nabla_1 \g(x_i,x_j)  - \beta \nab \g * \mu^t (x_i) \right|^2df_N^t,
\end{multline}
where
\begin{align}
u^t \coloneqq  \frac1{\beta}\nabla\log\mu^t +\nabla V+ \nabla \g \ast \mu^t
\end{align}
is the velocity field associated to the mean-field dynamics \eqref{meanfield}.

At first pass, the  second term on the right-hand side of \eqref{eq:MFEdiss}, which is nonpositive,  may be discarded, and,  assuming $\g$ is translation-invariant, the first term in the right-hand side can be controlled, for instance in Riesz cases \eqref{rieszcase}   via the second author's inequality from \cite{Serfaty2020} and its refinements and generalizations \cite{NRS2021,RS2022}, by the  modulated energy itself, allowing to close a Gr\"{o}nwall loop.  When $V=0$, this is what is done in \cite{BJW2019edp} and revisited in \cite{CdCRS2023,CdCRS2023a}.  More precisely, the following type of inequality is used: for any sufficiently regular vector field $v$ and any pairwise distinct $\XN\in (\R^\d)^N$ ,
\begin{multline}\label{commutator}
\left|\int_{(\R^\d)^2\setminus\triangle} (v(x)-v(y))\cdot \nabla_1\g(x,y) d\left(\frac1N\sum_{i=1}^N\delta_{x_i} - \mu\right)^{\otimes 2}(x,y)\right|\\ \le C \|v\|_{*} \( F_N(\XN, \mu) +  o_N(1)\),
\end{multline}
where $\|v\|_{*}$ is some homogeneous Sobolev norm of $v$ and $o_N(1)$ depends only on (and is increasing with respect to) the $L^\infty$ norm of $\mu$ and vanishes as $N\rightarrow\infty$. This inequality was first proven in full generality in \cite{Serfaty2020} for all Coulomb/super-Coulombic Riesz potentials, following a previous work for the $\d=2$ Coulomb case \cite{LS2018}. A sharp additive error $o_N(1)= O( \|\mu\|_{L^\infty}^{\frac{s}{\d} } N^{\frac{s}{\d}-1})$ with $\|v\|_{*} = \|\nabla v\|_{L^\infty}$ was proven in \cite{RS2022}, following earlier Coulomb results \cite{LS2018, Serfaty2023, Rosenzweig2021ne}. The estimate \eqref{commutator} was generalized to Riesz-like kernels in \cite{NRS2021}. For $\g$ satisfying $\left|(x-y)\cdot\nabla_1\g(x,y)\right| \leq C$, one may extract from \cite{JW2018}, as was done in \cite{BJW2020}, the averaged inequality
\begin{multline}\label{commutator'}
\left|\int_{(\R^\d)^N}\int_{(\R^\d)^2\setminus\triangle} (v(x)-v(y))\cdot \nabla_1 \g(x,y) d\left(\frac1N\sum_{i=1}^N\delta_{x_i} - \mu\right)^{\otimes 2}(x,y)df_N\right|\\
\leq \|\nabla v\|_{L^\infty}\left(C_1 H_N(f_N \vert \mu^{\otimes N}) + \frac{C_2}{N}\right).
\end{multline}

Let us now examine the nonpositive term in the right-hand side of \eqref{eq:MFEdiss}. We rewrite it 
as 
\begin{align}\label{rewriting}
&-\frac{1}{\beta^2 N}\int_{(\R^\d)^N}\sum_{i=1}^N \left|\nabla\log\paren*{\frac{f_N^t}{(\mu^t)^{\otimes N}}} + \frac{\beta}{N} \sum_{j\neq i}  \nabla_1 \g(x_i,x_j)  -  \nab \g * \mu^t (x_i) \right|^2df_N^t \nn\\
&= - \frac{1}{\beta^2 N} \int_{(\R^\d)^N }\left|\nab \log \frac{f_N^t}{\mathbb{Q}_{N,\beta} (\mu^t) } \right|^2 df_N^t = - \frac{1}{\beta^2 N} \int_{(\R^\d)^N}   \left| \nab \sqrt{ \frac{f_N^t}{\mathbb{Q}_{N,\beta} (\mu^t) } } \right|^2 d \mathbb{Q}_{N,\beta}(\mu^t) .
\end{align}
Indeed, one may check that by definition \eqref{defQ} of $\mathbb{Q}_{N,\be}(\mu)$,
\begin{align}
\nab_i \log \mathbb{Q}_{N,\beta}(\mu) = -\beta N \nab_i F_N(\XN, \mu) +  \nab \log \mu(x_i),
\end{align}
and  in view of the definition \eqref{defF} of $\Fr_N(\XN,\mu)$,
\begin{align}
\nab_i F_N (\XN,\mu) = \frac{1}{N^2}  \sum_{1\leq j\leq N: j\neq i} \nab_1 \g(x_i,x_j) - \frac{1}{N} \nab (\g* \mu) (x_i).
\end{align}
For any $f_N$ and any reference probability density $\mu$, we  call
the quantity 
\begin{equation}\label{modfish}
\frac{1}{N} \int_{(\R^\d)^N } \left |\nab \sqrt{   \frac{f_N}{\mathbb{Q}_{N,\beta} (\mu) } }\right|^2 d \mathbb{Q}_{N,\beta}(\mu)
\end{equation}
the {\it modulated Fisher information},  which is nothing but the normalized relative Fisher information $I_N(f_N \vert \mathbb{Q}_{N,\be}(\mu))$, and the relation \eqref{eq:MFEdiss} transforms into 
 \begin{multline}\label{MFEdiss2}
\frac{d}{dt}E_N(f_N^t, \mu^t)  \leq -\frac{1}{\beta^2 N}  \int_{(\R^\d)^N } \left |\nab \sqrt{   \frac{f_N^t}{\mathbb{Q}_{N,\beta} (\mu^t) } }\right|^2 d \mathbb{Q}_{N,\beta}(\mu^t)
\\
 -\frac{1}{2} \int_{(\R^\d)^N}\int_{(\R^\d)^2\setminus\triangle} (u^t(x)-u^t(y))\cdot \nabla_1\g(x,y) d\left(\frac1N\sum_{i=1}^N\delta_{x_i} - \mu^t\right)^{\otimes 2}(x,y)d f_N^t . \end{multline}

The goal is then to exploit a functional inequality relating the modulated Fisher information to the modulated free energy to take advantage of the negative term in \eqref{MFEdiss2}.

\begin{mydef}\label{def:ULSI}
We say that a family of probability measures $\{P_N\}_{N\geq 1}$ satisfies a {\it uniform logarithmic Sobolev inequality} (LSI) if there exists a constant  $C_{LS}>0$, such that for any $N\geq 1$ and $f \in C^1((\R^\d)^N)$, we have 
\begin{equation}\label{eq:LSI} 
\int_{(\R^\d)^N} f^2 \log \frac{f^2}{\int f^2 dP_N} dP_N \le C_{LS} \int_{(\R^\d)^N} |\nab f|^2 dP_N.
\end{equation}
Given data $(\g,V,\beta)$, we say that a \emph{uniform $\mu$-modulated LSI ($\mu$-LSI)} holds if the family of probability measures $\{\mathbb{Q}_{N,\beta}(\mu)\}_{N\geq 1}$ of the form \eqref{defQ} satisfies a uniform LSI.
\end{mydef}

Our main observation is that if $\mathbb{Q}_{N,\beta}(\mu)$ satisfies a uniform LSI, then applying \eqref{eq:LSI} to $f= \sqrt{ \frac{f_N }{\mathbb{Q}_{N,\beta} (\mu)}}$, with $f_N$ a probability density on $(\R^\d)^N$, we find 
\begin{equation}
\int_{(\R^\d)^N} \left|\nab \sqrt { \frac{f_N }{\mathbb{Q}_{N,\beta} (\mu)}  }\right|^2  d\mathbb{Q}_{N,\beta}(\mu) 
\ge \frac{1}{C_{LS} } \int_{(\R^\d)^N}   \log\Bigg(\frac{ \frac{f_N }{\mathbb{Q}_{N,\beta} (\mu)}    } {\int_{(\R^\d)^N } \frac{f_N }{\mathbb{Q}_{N,\beta} (\mu)}d\mathbb{Q}_{N,\beta}(\mu) }\Bigg)df_N .
\end{equation}
Using that $f_N$ is a probability density, we recognize on the right-hand side  $N H_N(f_N\vert \mathbb{Q}_{N,\beta}(\mu))$. In light of \eqref{rewritemodenergy}, we then have 
\begin{equation}\label{fishmoden}
\frac1N \int_{(\R^\d)^N} \left|\nab \sqrt { \frac{f_N }{\mathbb{Q}_{N,\beta} (\mu)}  }\right|^2  d\mathbb{Q}_{N,\beta}(\mu)
\ge  \frac{1}{C_{LS}} \(  \beta E_N(f_N, \mu) - \frac1N \log K_{N,\beta}(\mu)\).
\end{equation}
In other words,  a uniform LSI for $\mathbb{Q}_{N,\beta} (\mu)$ implies that the modulated Fisher information is bounded below by the modulated free energy  and an additive error that is $o_N(1)$ assuming smallness of free energy. If \eqref{fishmoden} holds for all $\mu^t$ along the flow, then  it can be inserted into \eqref{MFEdiss2} to obtain an exponential decay of the modulated  free energy, provided \eqref{commutator} or \eqref{commutator'} holds. 

 In \cite{GlBM2021}, in the context of conservative dynamics on the torus $\T^d$ (see remarks at the end of \cref{ssec:introMR}), a uniform LSI is used in the context of the relative entropy method \cite{JW2018}. In that method,  one differentiates in time $H_N(f_N^t\vert (\mu^t)^{\otimes N})$ instead of \eqref{rewritemodenergy}, leading to a Fisher information relative to the reference measure $(\mu^t)^{\otimes N}$ instead of $\mathbb{Q}_{N,\beta}(\mu^t)$. Proving the needed uniform LSI holds is straightforward, as it follows from upper and lower bounds on $\mu^t$ (a consequence of maximum principle and only possible on compact domains) and the Holley-Stroock perturbation lemma. See also \cite{LlF2023} for a similar idea applied to the hierarchal relative entropy method of \cite{Lacker2023}.

\subsection{Main result}\label{ssec:introMR}
 To present the main result of this note, we list some assumptions that we make on the potential $\g: (\R^\d)^2\rightarrow [-\infty,\infty]$. We will explain below specific cases in which these assumptions hold.

\begin{enumerate}[(i)]
\item\label{assmp1}
$\g \in C^2((\R^\d)^2\setminus\triangle)$ is symmetric and for some $s<\d$, satisfies
\begin{align}
|\g(x,y)| \leq C\begin{cases}1+\left|\log|x-y|\right|, & {s=0} \\ 1+ |x-y|^{-s}, & {s>0} \end{cases}
\end{align}
for some constant $C>0$.
\item\label{assmp2'}
There exists a constant $C_\beta \in [0,\frac{1}{\beta})$ such that for any $f_N \in \P_{ac}((\R^\d)^N)$ and $\mu \in \P(\R^\d) \cap L^\infty(\R^\d)$, with $\int_{\R^\d}\log(1+|x|)d\mu(x)<\infty$ if $s=0$,
\begin{align}
\mathbb{E}_{f_N}\left[\Fr_N(\XN,\mu)\right] \geq -C_\beta H_N(f_N \vert \mu^{\otimes N}) - o_N(1),
\end{align}
where $o_N(1)$ only depends (in an increasing fashion) on $\mu$ through $\|\mu\|_{L^\infty}$.
\item\label{assmp3}
There exist constants $C_{RE},C_{ME}\geq 0$, such that
\begin{multline}\label{commutator''}
\left|\int_{(\R^\d)^N}\int_{(\R^\d)^2\setminus\triangle} (v(x)-v(y))\cdot \nabla_1\g(x,y) d\left(\frac1N\sum_{i=1}^N\delta_{x_i} - \mu\right)^{\otimes 2}(x,y)df_N\right|\\
\leq \|v\|_{*}\left(C_{RE}H_N(f_N \vert \mu^{\otimes N}) + C_{ME}\mathbb{E}_{f_N}\left[\Fr_N(\XN,\mu)\right] + o_N(1)\right)
\end{multline}
for all pairwise distinct configurations $\XN\in(\R^\d)^N$, densities $f_N\in \P_{ac}((\R^\d)^N)$ and $\mu\in \P(\R^\d)\cap L^\infty(\R^\d)$, and continuous vector fields $v$ with finite homogeneous Sobolev norm $\|\cdot\|_*$ of some order.
\end{enumerate}

\begin{remark}\label{rem:assmps}
Assumption \eqref{assmp1} is to ensure that all energy expressions are well-defined and that all differential identities can be justified. Assumption \eqref{assmp2'} ensures that the modulated energy does overwhelm the relative entropy, which is not \emph{a priori} forbidden, since we make no sign assumptions on $\g$. Since $C_{\beta}<\frac{1}{\beta}$, it ensures that the modulated free energy is nonnegative up to $o_N(1)$ error. In fact, it shows that the modulated free energy controls the relative entropy. 
\end{remark}

Let us introduce the quantity
\begin{equation}\label{eq:EcNdef}
 \mathcal{E}_N^t \coloneqq E_N(f_N^t,\mu^t) + o_N^t(1)
 \end{equation}
 as a substitute for the modulated free energy. The additive error $o_N^t(1)$ is a constant multiple of the maximum of the additive errors in assumptions \eqref{assmp2'}, \eqref{assmp3} and ensures that $\mathcal{E}_N^t\geq 0$, which allows to perform a Gr\"onwall argument on this quantity. It depends only on $\mu^t$ through the $L^\infty$ norm, hence the $t$ superscript, and is increasing in this dependence. Also, it is easier to write the statements with $\mathcal{E}_N^t$, as these additive constants appear as the errors $o_N(1)$ in \eqref{commutator}. 
 
\begin{thm}\label{thm:main}
Let $\beta>0$. Assume that equation \eqref{meanfield} admits a solution $\mu \in C([0,\infty), \mathcal{P}(\R^\d) \cap L^\infty(\R^\d))$, such that  $\|\mu^t\|_{L^\infty}$ is bounded uniformly in $t$ and $\nab u^t\in L^\infty$ locally uniformly in $t$. If $s=0$, further assume that $\int_{\R^\d}\log(1+|x|)d\mu^t<\infty$ for every $t\geq 0$. If $\mathbb{Q}_{N,\beta}(\mu^t)$ satisfies a uniform LSI with constant $C_{LS}>0$ for every $t\geq 0$, then 
\begin{multline}\label{eq:main}
\forall t\geq 0, \qquad \Ec_N^t \leq  e^{-\frac{4 t}{ \beta C_{LS}}+\int_0^t\frac{\Cs\| u^{\tau}\|_{*}}{2}d\tau}\Ec_N^0  \\
+e^{-\frac{4 t}{\beta C_{LS}}+\int_0^t\frac{\Cs\|u^\tau\|_{*}}{2}d\tau}\int_0^t e^{\frac{4\tau}{\beta C_{LS}}-\int_0^\tau\frac{\Cs\| u^{\tau'}\|_{*}}{2}d\tau'}\left[\dot{o}_N^\tau+ \frac{4}{\beta C_{LS}}\Bigg(o_N^\tau(1) - \frac{\log K_{N,\be}(\mu^\tau) }{\beta N} \Bigg)\right]d\tau,
\end{multline}
where $K_{N,\be}(\mu^\tau)$ is as in \eqref{defKNbe}, $o_N^\tau(1)$ is as above, and $\dot{o}_N^\tau(1)$ denotes the derivative of $o_N^\tau(1)$ with respect to time.
\end{thm}
We see here that provided  $\int_0^\infty\|u^\tau\|_{*}d\tau <\infty$, the first term on the right-hand side converges exponentially fast to $0$ as $t\rightarrow\infty$, while the second term is $o_N(1)$ uniformly bounded in $t$, assuming $\log K_{N,\be}(\mu^\tau)=o(N)$ uniformly in $\tau$ and that $\int_0^\infty |\dot{o}_N^\tau(1)|d\tau<\infty$, by the fundamental theorem of calculus and our assumption that $\|\mu^t\|_{L^\infty}$ is uniformly bounded. Since $\mathcal E_N$ differs from $E_N$ only by additive constants which are $o_N(1)$, and the modulated free energy $E_N$ controls the relative entropy $H_N$, as explained in \cref{rem:assmps}, it follows that the estimate \eqref{eq:main} implies entropic generation of chaos and also gives a uniform-in-time propagation of chaos if the initial data is such that $\mathcal E_N^0=o_N(1)$. In the next subsection, we give cases of interest to which \cref{thm:main} applies.

Generation of chaos for potentials $\g$ with $\nabla\g$ in $L^\infty$, which does not allow for singular potentials, was shown in \cite{LlF2023}, with a rate of convergence in $N$ that is sharp for relative entropy, under smallness assumptions on $\beta$. In \cite{GlBM2021}, a generation of chaos result was shown for \emph{conservative} dynamics (replace $\nabla$ with $\M\nabla$ for an antisymmetric matrix $\M$) with $\g$ having a log-type singularity. Both \cite{LlF2023, GlBM2021} are restricted to the torus $\T^\d$. A weaker generation of chaos result in $2$-Wasserstein distance was shown in \cite{GlBM2023} for the Riesz case on $\R$ with uniformly convex confinement via coupling methods. We mention that convergence in relative entropy implies convergence in $W_2$ by a theorem of Otto-Villani \cite{OV2000}.

\begin{remark}\label{rem:decay}
The long-time analysis of equation \eqref{meanfield} that allows to show in the Riesz case that $\int_0^\infty \|\nabla u^\tau\|_{L^\infty}d\tau < \infty$ and $K_{N,\be}(\mu^\tau)=o(N)$ uniformly in $\tau$ is the subject of forthcoming work with J. Huang \cite{HRS2022}.  In fact, this work shows that solutions converge as $t\rightarrow\infty$ to the thermal equilibrium $\mu_\beta$ in a strong sense at a quantifiable rate and even covers the case of $\R^\d$ without confinement, which has been an outstanding problem.
\end{remark}

\begin{remark}\label{rem:torus}
One could also consider the periodic setting $\T^\d$, as in \cite{BJW2019edp, BJW2020, GlBM2021, CdCRS2023}. But the case of $\R^{\d}$ is mathematically more interesting.
\end{remark}

In the case where $\d=1$ and $\g(x)=-\log|x|$ or $|x|^{-s}$ for $s\in (0,1)$, $V$ is a $C^2$ uniformly convex potential (e.g., $V(x)=|x|^2$), and $\mu$ is a probability density which is not too far from the thermal equilibrium $\mu_\beta$, we are able to verify a uniform $\mu$-modulated LSI. The general $\d$-dimensional Riesz case is challenging: it is at least as difficult as the uniform LSI for $\mathbb{P}_{N,\be}^V$, which is a well-known open problem.

\subsection{Applications}\label{ssec:introApps}
We can give a more precise form of the estimate \eqref{eq:main} in the repulsive singular Riesz case \eqref{rieszcase} so that $(-\Delta)^{\frac{\d-s}{2}}\g = \mathsf{c}_{\d,s}\delta_0$. One has that 
\begin{align}\label{eq:MElbRiesz}
\Fr_N(\ux_N,\mu) \geq -\begin{cases} \frac{\log(N\|\mu\|_{L^\infty})}{2N\d}\indic_{s=0} + \Cs\|\mu\|_{L^\infty}^{\frac{s}{\d}}N^{\frac{s}{\d}-1}, & {s\geq\d-2} \\ \frac{\mathsf{C}\log(N\|\mu\|_{L^\infty})}{N}\indic_{s=0} + \mathsf{C}\|\mu\|_{L^\infty}^{\frac{s}{\d}} N^{-\frac{2(\d-s)}{2(\d-s)+s(\d+2)}}, &{s<\d-2}.\end{cases}
\end{align}
Here, $\mathsf{C}>0$ is an absolute constant. The additive errors for the sub-Coulomb case $s<\d-2$ are expected to be suboptimal, while they are sharp in the Coulomb/super-Coulomb case $s\geq\d-2$.\footnote{The $L^\infty$ condition here---and by implication, the $L^\infty$ condition in \cref{thm:main}---can be relaxed quite a bit (e.g., see \cite{Rosenzweig2022,Rosenzweig2022a}) at the cost of increasing the additive errors; but we will not concern ourselves with such generality.} For details, we refer to \cite{RS2021} in the case $s<\d-2$ and \cite{CdCRS2023,RS2022} in the case $s\geq \d-2$. In particular, \eqref{eq:MElbRiesz} shows that
\begin{align}
E_N(f_N,\mu) \geq \frac{1}{\beta}H_N(f_N \vert \mu^{\otimes N}) -\begin{cases} \frac{\log(N\|\mu\|_{L^\infty})}{2N\d}\indic_{s=0} + \Cs\|\mu\|_{L^\infty}^{\frac{s}{\d}}N^{\frac{s}{\d}-1}, & {s\geq\d-2} \\ \frac{\mathsf{C}\log(N\|\mu\|_{L^\infty})}{N}\indic_{s=0} + \mathsf{C}\|\mu\|_{L^\infty}^{\frac{s}{\d}} N^{-\frac{2(\d-s)}{2(\d-s)+s(\d+2)}}, &{s<\d-2}.\end{cases}
\end{align}
 We take
\begin{equation}\label{eq:EcNdefRiesz}
 \mathcal{E}_N^t \coloneqq  E_N(f_N^t,\mu^t) + \begin{cases} \frac{\log(N\|\mu^t\|_{L^\infty})}{2N\d}\indic_{s=0} + \Cs\|\mu^t\|_{L^\infty}^{\frac{s}{\d}}N^{\frac{s}{\d}-1}, & {s\geq\d-2} \\ \frac{\mathsf{C}\log(N\|\mu^t\|_{L^\infty})}{N}\indic_{s=0} + \mathsf{C}\|\mu^t\|_{L^\infty}^{\frac{s}{\d}} N^{-\frac{2(\d-s)}{2(\d-s)+s(\d+2)}}, &{s<\d-2}\end{cases}
 \end{equation}
The estimate \eqref{commutator''} holds with $C_{RE}=0$,
\begin{align}
\|v\|_{*} = \begin{cases} \|\nabla v\|_{L^\infty}, & {s\geq \d-2} \\ \|\nabla v\|_{L^\infty} + \|(-\Delta)^{\frac{\d-s}{4}}v\|_{L^{\frac{2\d}{\d-2-s}}}, & {s<\d-2}, \end{cases}
\end{align}
and
\begin{align}
o_N^t(1) = \begin{cases}\frac{\log(N\|\mu^t\|_{L^\infty})}{2N\d}\indic_{s=0} + \Cs\|\mu^t\|_{L^\infty}^{\frac{s}{\d}}N^{\frac{s}{\d}-1}, & {s\geq\d-2}\\ \\ \|(-\Delta)^{\frac{s+1-\d}{2}}\mu^t\|_{L^\infty} N^{-\frac{s+1+\frac{(2(\d-s)}{\d+2}}{\left(s+\frac{(2(\d-s)}{\d+2}\right)(1+s)}} + \|\mu^t\|_{L^\infty}^{\frac{2+s}{\d+2}}N^{-\frac{\frac{(2(\d-s)}{\d+2}}{\left(s+\frac{(2(\d-s)}{\d+2}\right)(1+s)}}, & {s<\d-2}. \end{cases}
\end{align}
In the attractive log case on the torus $\T^\d$, it is shown in \cite{CdCRS2023a} (building on \cite{BJW2020}) that there exists $\beta_\d >0$ such that for any $0\leq\beta<\beta_\d$, there are constants $\ep_\beta\in (0,1)$ and $C_\beta>0$ such that 
\begin{align}
\beta\mathbb{E}_{f_N}\left[\Fr_N(\XN,\mu) \right] \leq \ep_\beta H_N(f_N \vert \mu^{\otimes N}) + \frac{C_\beta}{N}.
\end{align}
Therefore, assumption \eqref{assmp2'} is satisfied. The conjectured optimal value of $\beta_\d$ is $2\d$ (e.g., $\beta_\d=4$ in the $\d=2$ case, which corresponds to the Patlak-Keller-Segel model). It is shown in \cite{CdCRS2023a} that that $\beta_\d\leq 2\d$ and further that if $\d=2$, then $\beta_\d=2\d$ provided one restricts to densities $\mu$ sufficiently close to the uniform measure.

For potentials $\g:(\R^\d)^2\rightarrow\R$ that are continuous along the diagonal, one can skip the renormalization and simply define
\begin{align}
\Fr_N(\XN,\mu) = \int_{(\R^\d)^2}\g(x,y)d\left(\frac1N\sum_{i=1}^N\delta_{x_i}-\mu\right)(x)d\left(\frac1N\sum_{i=1}^N\delta_{x_i}-\mu\right)(y).
\end{align}
If $\g$ is repulsive in the sense that $\g(x,y)$ is the integral kernel of a positive semidefinite operator on the space of finite Borel measures, as in the case of the equations used for neural networks parameters evolution \cite{MMM2019,RvE2022,CB2018}, then $\Fr_N(\XN,\mu)\geq 0$.

Continuing to assume that $\g$ is continuous at the origin, but dropping the repulsive assumption, we may use the Donsker-Varadhan lemma to estimate
\begin{align}
\mathbb{E}_{f_N}\left[\Fr_N(\XN,\mu)\right] \leq \frac{1}{\eta}\left(H_N(f_N\vert\mu^{\otimes N})  +\frac{1}{N}\log\mathbb{E}_{\mu^{\otimes N}}\left[e^{N\eta\Fr_N(\XN,\mu)}\right]\right)
\end{align}
for any $\eta>0$. If $\g\in L^\infty$, then one may use \cite[Theorem 4]{JW2018} (see also \cite[Section 5]{LLN2020} for a simpler proof) with
\begin{align}
\phi(x,z) \coloneqq \left(\g(x,z) - \int_{\R^\d}\g(x,y)d\mu(y) - \int_{\R^\d}\g(y,z)d\mu(y) + \int_{(\R^\d)^2}\g(y,y')d\mu(y)d\mu(y')\right).
\end{align}
The conclusion is that if $\sqrt{C_0}\eta\|\phi\|_{L^\infty}<1$, where $C_{0}$ is a universal constant, then
\begin{align}
\log\mathbb{E}_{\mu^{\otimes N}}\left[e^{N\eta\Fr_N(\XN,\mu)}\right] \leq \log\left(\frac{2}{1-C_0\eta^2\|\phi\|_{L^\infty}^2}\right).
\end{align}
Replacing $F_N$ by $-F_N$ and repeating the preceding reason, we then find that
\begin{align}
\left|\mathbb{E}_{f_N}\left[\Fr_N(\XN,\mu)\right]\right| \leq \frac{1}{\eta}H_N(f_N\vert\mu^{\otimes N}) + \frac{1}{\eta N}\log\left(\frac{2}{1-C_0\eta^2\|\phi\|_{L^\infty}^2}\right).
\end{align}
If $\frac{1}{\beta}>\frac{1}{\sqrt{C_0}\|\phi\|_{L^\infty}}$, then we may choose $\frac{1}{\eta}\in (\frac{1}{\sqrt{C_0}\|\phi\|_{L^\infty}}, \frac{1}{\beta})$, implying that assumption \eqref{assmp2'} holds. 

\subsection{Advantages of modulated LSI over LSI}\label{sec1.5}
Let us explain the advantage of a uniform modulated LSI over merely a uniform LSI for $\mathbb{P}_{N,\be}$. For simplicity, let us assume that $\g$ is translation-invariant. Ignoring regularity questions,
\begin{align}
\frac{d}{dt}H_N\left(f_N^t \vert \mathbb{P}_{N,\be}\right) = -\frac{1}{\beta} I_N(f_N^t \vert \mathbb{P}_{N,\beta}),
\end{align}
where we recall that $I_N$ is the normalized relative Fisher information. If there is a uniform LSI constant $C_{LS}$ for $\mathbb{P}_{N,\be}$, then by Gr\"{o}nwall's lemma,
\begin{align}
H_N\left(f_N^t \vert \mathbb{P}_{N,\be}\right) \leq e^{-\frac{t}{C_{LS}\beta}}H_N\left(f_N^0 \vert \mathbb{P}_{N,\be}\right).
\end{align}
By subadditivity of relative entropy and Pinsker's inequality, for any fixed $1\leq k\leq N$,
\begin{align}
\left\|f_{N,k}^t - \mathbb{P}_{N,\be}^{(k)}\right\|_{TV}^2 \leq 2k e^{-\frac{t}{C_{LS}\beta}} H_N\left(f_N^0 \vert \mathbb{P}_{N,\be}\right).
\end{align}
If $\mu^t$ is a solution of the mean-field evolution \eqref{meanfield}, then
\begin{align}\label{eq:dissmfFE}
\frac{d}{dt}\left[\Ec_{\beta}(\mu^t) - \Ec_{\beta}(\mu^\beta) \right] = -\int_{\R^\d}\left|\frac{1}{\beta}\nabla\log\mu^t + \nabla\g\ast\mu^t + \nabla V\right|^2 d\mu^t,
\end{align}
where $\mathcal{E}_{\beta}$ is the mean-field free energy as defined in \eqref{MFenergy}. 
One may check by direct computation that
\begin{align}
\lim_{N\rightarrow\infty} \frac{1}{\beta} H_N(\mu^{\otimes N} \vert \mathbb{P}_{N,\beta}) = \mathcal{E}_{\beta}(\mu) - \mathcal{E}_{\beta}(\mu_\beta)
\end{align}
and
\begin{align}
\lim_{N\rightarrow\infty} \frac{1}{\beta}I_N(\mu^{\otimes N} \vert \mathbb{P}_{N,\beta}) = \beta\int_{\R^\d}\left|\frac{1}{\beta}\log\mu +\nabla V + \nabla\g\ast\mu\right|^2d\mu,
\end{align}
which together with the uniform LSI for $\mathbb{P}_{N,\beta}$ imply the infinite-volume LSI
\begin{align}
\beta\left[ \mathcal{E}_{\beta}(\mu) - \mathcal{E}_{\beta}(\mu_\beta)\right] = \lim_{N\rightarrow\infty} H_N(\mu^{\otimes N} \vert \mathbb{P}_{N,\beta}) &\leq \lim_{N\rightarrow\infty} C_{LS}I_N(\mu^{\otimes N} \vert \mathbb{P}_{N,\beta}) \nn\\
&= C_{LS}\beta^2\int_{\R^\d}\left|\frac{1}{\beta}\log\mu +\nabla V + \nabla\g\ast\mu\right|^2d\mu.
\end{align}
Inserting this inequality into the right-hand side of \eqref{eq:dissmfFE} and applying Gr\"onwall again,
\begin{align}
\left[\Ec_{\beta}(\mu^t) - \Ec_{\beta}(\mu^\beta) \right] \leq e^{-\frac{t}{C_{LS}\beta}}\left[\Ec_{\beta}(\mu^0) - \Ec_{\beta}(\mu^\beta) \right].
\end{align}
Using \eqref{charactem} and direct computation, one may also check that
\begin{align}
\mathcal{E}_{\beta}(\mu) - \mathcal{E}_\beta(\mu_\beta) = \frac{1}{\beta}\int_{\R^\d}\log\left(\frac{\mu}{\mu_\beta}\right)d\mu + \frac12\int_{(\R^\d)^2}\g(x-y)d(\mu-\mu_\beta)^{\otimes 2}(x,y).
\end{align}
Assuming, say, that $\hat{\g}\geq 0$, we may discard the potential energy term and then apply Pinsker's inequality again to obtain
\begin{align}
\left\|\mu^t-\mu_\be\right\|_{TV}^2 \leq 2\beta e^{-\frac{t}{C_{LS}\beta}}\left[\Ec_{\beta}(\mu^0) - \Ec_{\beta}(\mu^\beta) \right].
\end{align}
Considering just the case $k=1$ to simplify the analysis, we have by triangle inequality that
\begin{align}
\left\|f_{N,1}^t - \mu^t\right\|_{TV} &\leq \left\|f_{N,1}^t - \mathbb{P}_{N,\beta}^{(1)}\right\|_{TV} + \left\| \mathbb{P}_{N,\beta}^{(1)} - \mu_\beta \right\|_{TV} + \left\| \mu^t- \mu_\beta \right\|_{TV}\nn\\
&\leq \sqrt{2e^{-\frac{t}{C_{LS}\beta}} H_N\left(f_N^0 \vert \mathbb{P}_{N,\be}\right)} + \sqrt{2\beta e^{-\frac{t}{C_{LS}\beta}}\left(\Ec_{\beta}(\mu^0) - \Ec_{\beta}(\mu^\beta) \right)} + \left\| \mathbb{P}_{N,\beta}^{(1)} - \mu_\beta \right\|_{TV}. \label{eq:TIrhs}
\end{align}
Supposing that\footnote{Such a bound is known (with a sharp estimate for $o_{N}(1)$), for instance, in the high-temperature case where $\g$ has bounded gradient \cite{Lacker2022Gibbs}.}
\begin{align}
\left\|\mathbb{P}_{N,\beta}^{(1)}-\mu_\beta\right\|_{TV} = o_{N}(1),
\end{align}
the right-hand side of \eqref{eq:TIrhs} tends to zero as $t\rightarrow\infty$ and $N\rightarrow\infty$. But this estimate does not imply propagation of chaos, even locally in time, as the second term does not vanish as $N\rightarrow\infty$. To address this unsatisfactory feature, one would also need a local-in-time estimate with which to interpolate, say, of the form
\begin{align}
\left\|f_{N,1}^t - \mu^t\right\|_{TV} \leq e^{Ct}o_N(1),
\end{align}
where $o_N(1)$ vanishes as $N\rightarrow\infty$ assuming some form of chaos for the initial data.

The above described argument is rather inefficient. We had to pass from relative entropy to a genuine metric, total variation distance, to implement this triangle inequality argument. In doing so, one loses the optimality of the rate in $N$ \cite{Lacker2023}. Moreover, by trying to balance $t$ and $N$, the rate of convergence further deteriorates. In contrast, a uniform modulated LSI addresses propagation/generation of chaos in one swoop, because it is dynamic: it not only depends on $N$ but also allows for dependence on $t$ through the flowing of $\mu$ according to \eqref{meanfield}.

We mention that this classical triangle inequality/interpolation idea was used in \cite{DGPS2023} for energies with regular interactions, except with total variation distance replaced by $2$-Wasserstein distance, which works just as well since LSI implies a Talagrand inequality \cite{OV2000}. Though, only a statement of uniform-in-time propagation of chaos (with suboptimal rate), as opposed to generation of chaos, is presented in \cite{DGPS2023}.


\subsection{Organization of the paper}
Let us conclude the introduction with some remarks on the organization of the body of the paper. In \cref{sec:ProofMT}, we give the details of the proof of the main result, \cref{thm:main}. Then in \cref{sec:LSI}, we turn to proving that a uniform modulated LSI holds in the log/Riesz case for a certain class of densities $\mu$ in dimension $\d=1$.

\subsection{Acknowledgments}
The authors thank Djalil Chafa\"{i} for helpful discussion and references. The second author also acknowledges the Fondation Sciences Math\'ematiques de Paris and PSL Research University who supported her visit to ENS-PSL, where this work was completed.

\section{Proof of the main theorem}\label{sec:ProofMT}
Applying the uniform LSI for $\mathbb{Q}_{N,\beta}(\mu^t)$ to the first term in the right-hand side of \eqref{MFEdiss2} via \eqref{fishmoden}, we find, abbreviating $K_N^t \coloneqq K_{N,\beta}(\mu^t)$, 
\begin{multline}\label{eq:MFEdissLS}
\frac{d}{dt}E_N(f_N^t, \mu^t)  \leq  -\frac{1}{2} \int_{(\R^\d)^N}\int_{(\R^\d)^2\setminus\triangle} (u^\tau(x)-u^\tau(y))\cdot \nabla_1\g(x,y) d\left(\frac1N\sum_{i=1}^N\delta_{x_i} - \mu^t\right)^{\otimes 2}(x,y)d f_N^t \\
-\frac{4}{\beta C_{LS}}\left(E_N(f_N^t,\mu^t) + \frac{\log K_N^t}{\beta N}\right).
\end{multline}
 Under the assumption \eqref{assmp3}, we have
\begin{multline}
\int_{(\R^\d)^N}\left|\int_{(\R^\d)^2\setminus\triangle} (u^\tau(x)-u^\tau(y))\cdot \nabla_1\g(x,y) d\left(\frac1N\sum_{i=1}^N\delta_{x_i} - \mu^t\right)^{\otimes 2}(x,y)\right| d f_N^t \\
\leq   \|u^t \|_{*}\left(C_{RE}H_N(f_N^t\vert (\mu^t)^{\otimes N}) +  C_{ME}\mathbb{E}_{f_N^t}\left[\Fr_N(\XN,\mu^t)\right] + o_N^t(1)\right). \label{eq:MEfvar}
\end{multline}
If $\frac{C_{RE}}{C_{ME}}\leq \frac{1}{\beta}$, then since $H_N(f_N^t \vert (\mu^t)^{\otimes N})\geq 0$, we may assume without loss of generality that $\frac{C_{RE}}{C_{ME}}=\frac{1}{\beta}$. If $\frac{C_{RE}}{C_{ME}}> \frac{1}{\beta}$, then using assumption \eqref{assmp2'},
\begin{align}
C_{ME}\mathbb{E}_{f_N^t}\left[\Fr_N(\XN^t,\mu^t)\right]  &\leq C_{ME}\left( \mathbb{E}_{f_N^t}\left[\Fr_N(\XN,\mu^t)\right] + C_{\beta}H_N(f_N^t\vert (\mu^t)^{\otimes N}) + o_N^t(1)\right) \nn\\
&\leq C_{ME}'\left( \mathbb{E}_{f_N^t}\left[\Fr_N(\XN,\mu^t)\right] + C_{\beta}H_N(f_N^t\vert (\mu^t)^{\otimes N}) + o_N^t(1)\right)
\end{align}
for any $C_{ME}'\geq C_{ME}$. So, choosing $C_{ME}'$ sufficiently large so that $\frac{C_{RE}}{C_{ME}'} + C_{\beta} \leq \frac{1}{\beta}$ (remember that $C_{\beta}<\frac{1}{\beta}$ by assumption), we see that in all cases,
\begin{multline}
\int_{(\R^\d)^N}\left|\int_{(\R^\d)^2\setminus\triangle} (u^\tau(x)-u^\tau(y))\cdot \nabla_1\g(x,y) d\left(\frac1N\sum_{i=1}^N\delta_{x_i} - \mu^t\right)^{\otimes 2}(x,y)\right| d f_N^t \\
\leq \mathsf{C}\|u^t \|_{*}\left(\frac{1}{\beta}H_N(f_N^t \vert (\mu^t)^{\otimes N}) +\mathbb{E}_{f_N^t}\left[\Fr_N(\XN,\mu^t)\right]  + o_N^t(1)\right),
\end{multline}
for some constant $\mathsf{C}>0$. To establish a Gr\"{o}nwall relation, we use the quantity \eqref{eq:EcNdef}. We see from combining \eqref{eq:MFEdissLS} and \eqref{eq:MEfvar} that
\begin{align}
\frac{d}{dt}\mathcal{E}_N^t &\leq -\frac{4}{\beta C_{LS}}\left(E_N(f_N^t,\mu^t) + \frac{\log K_N^t}{\beta N}\right) + \frac{\Cs}{2}\| u^t\|_{*}\Ec_N^t {+ \dot{o}_N^t(1)}\nn\\
&=\left(-\frac{4}{\beta C_{LS}}+\frac{\Cs\| u^t\|_{*}}{2}\right)\Ec_N^t + \frac{4}{\beta C_{LS}}\left(o_N^t(1)- \frac{\log K_N^t}{\beta N} \right).
\end{align}
{Recall that $\dot{o}_N^t(1)$ denotes the time derivative.} Multiplying both sides by $e^{\int_0^t(\frac{4}{\beta C_{LS}}-\frac{\Cs\| u^{\tau'}\|_{*}}{2})d\tau'}$, we obtain
\begin{equation}
\frac{d}{dt}\left[e^{\int_0^t(\frac{4}{\beta C_{LS}}-\frac{\Cs\| u^{\tau'}\|_{*}}{2})d\tau'} \mathcal{E}_N^t \right] \leq e^{\int_0^t(\frac{4}{\beta C_{LS}}-\frac{\Cs\| u^{\tau'}\|_{*}}{2})d\tau'}\left[{\dot{o}_N^t(1) +} \frac{4}{\beta C_{LS}}\left(o_N^t(1) - \frac{\log K_N^t}{\beta N} \right)\right].
\end{equation}
Now using the fundamental theorem of calculus followed by a little rearrangement,
\begin{multline}\label{eq:EcNtfin}
\Ec_N^t \leq  e^{-\frac{4 t}{\beta C_{LS}}+\int_0^t\frac{\Cs\|u^{\tau}\|_{*}}{2}d\tau}\Ec_N^0\\
 +e^{-\frac{4 t}{\beta C_{LS}}+\int_0^t\frac{\Cs\|u^{\tau}\|_{*}}{2}d\tau}\int_0^t e^{\frac{4\tau}{\beta C_{LS}}-\int_0^\tau\frac{\Cs\| u^{\tau'}\|_{*}}{2}d\tau'}\left[{\dot{o}_N^\tau(1)} + \frac{4}{\beta C_{LS}} \left(o_N^\tau(1) - \frac{\log K_N^\tau}{\beta N} \right)\right]d\tau.
\end{multline}
This gives the estimate \eqref{eq:main} and therefore completes the proof of \cref{thm:main}.



\section{Uniform LSI for $\d=1$ Riesz case}\label{sec:LSI}
We show in this section that a uniform modulated LSI holds in the $\d=1$ repulsive Riesz case \eqref{rieszcase}
for uniformly convex confinement $V$. Using the notation from the introduction,
\begin{align}\label{eq:HNLSI}
\mathcal{H}_N(\XN) \coloneqq \sum_{i=1}^N V(x_i) + \frac{1}{N}\sum_{1\leq i<j\leq N}\g(x_j-x_i).
\end{align}

\begin{remark}\label{rem:C2}
In fact, the proof will show that a modulated LSI holds for any interaction potential $\g$ which is convex or for any $C^2$ interaction potential with $\|\g\|_{\dot{C}^2}$ sufficiently small depending on the convexity of $V$. We leave the details as an exercise for the reader. We expect that one could generalize further by following the proof of Zegarlinski's theorem \cite{Zegarlinski1992}, as used to show uniform LSIs in \cite{GLWZ2022}, or the two-scale approach of \cite{GOVW2009}, but will not pursue this.
\end{remark}

\begin{remark}\label{rem:Kuramoto}
As explained in \cite{Rosenzweig2023Kura}, the approach of \cite{BB2019spin} implies the LSI up to the critical inverse temperature for the Gibbs measure of the mean-field classical XY/$O(2)$/planar rotator/Kuramoto model, whose energy is far from convex. It is straightforward to adapt the reasoning of \cref{ssec:LSIQ1D} to obtain a modulated LSI for $\mu$ close enough to $\mu_\beta=1$.
\end{remark}

\subsection{Uniform LSI for $\mathbb{P}_{N,\be}^V$}\label{ssec:LSIP1D}
Following Chafa\"{i}-Lehec \cite{CL2020},\footnote{Strictly speaking, \cite{CL2020} considers the $\d=1$ $\log$ case; but the argument works with trivial modification in the general Riesz case. Furthermore, Chafa\"{i}-Lehec present more than one proof; but we choose to highlight the one based on Caffarelli's contraction theorem.} we present the LSI for the Gibbs measure $\mathbb{P}_{N,\be}^V$ in the $\d=1$ Riesz case with uniformly convex confinement $V$. This is a warm-up for proving the modulated LSI in the next subsection.

\begin{prop}\label{prop:1DRieszLSI}
Let $V: \R\rightarrow \R$ be $\ka$-convex for some $\ka>0$. For $\be> 0$, the probability measure $\mathbb{P}_{N,\be}^V$ has LSI constant $\frac{2}{\beta\ka}$.
\end{prop}
\begin{proof}
As $V$ is fixed, we omit the superscript in $\mathbb{P}_{N,\be}^V$ in what follows. By exchangeability, it suffices to restrict to the \emph{Weyl chamber}\footnote{This ability to order is, of course, a special feature of the one-dimensional setting.} $\Delta_N \coloneqq \{\XN\in \R^N : x_1\leq \cdots \leq x_N\}$. More precisely, define
\begin{align}
\tl{\g}(x) \coloneqq \begin{cases}\g(x), & {x>0} \\ \infty, & x\leq 0 \end{cases} \qquad \text{and} \qquad \tl{\mathcal{H}}_N(\XN) \coloneqq \sum_{i=1}^N V(x_i) + \frac{1}{N}\sum_{1\leq i<j\leq N}\tl\g(x_j-x_i),
\end{align}
and $d\tl{\mathbb{P}}_{N,\be} =\frac{ e^{-\beta\tl{\mathcal{H}}_N}}{\tl{Z}_{N,\be}}dX_N$. Since
\begin{align}
\int_{\R^N} e^{-\beta\mathcal{H}_N}d\XN = N!\int_{\D_N}e^{-\beta\mathcal{H}_N}d\XN,
\end{align}
it follows that if $\varphi$ is invariant under permutation of coordinates, then
\begin{align}
\int_{\R^N}\varphi^2\log\left(\frac{\varphi^2}{\int\varphi^2d\tl{\mathbb{P}}_{N,\be}}\right)d\tl{\mathbb{P}}_{N,\be}  = \int_{\R^N}\varphi^2\log\left(\frac{\varphi^2}{\int\varphi^2d\tl{\mathbb{P}}_{N,\be} }\right)d\tl{\mathbb{P}}_{N,\be}.
\end{align}
So, $\tl{\mathbb{P}}_{N,\be}$ has LSI constant $C_{LS}$ if and only if $\mathbb{P}_{N,\be}$ has LSI constant $C_{LS}$. Going forward, we drop the $\tl{}$ superscript in $\tl\g,\tl{\mathcal{H}}_N,\tl{\mathbb{P}}_{N,\be}$. 

Assuming that $V$ is $\ka$-convex, for some $\ka>0$, we claim that $\mathcal{H}_N$ is $\ka$-convex. Indeed, let $\XN, \YN\in\Delta_N$ and $\rho\in (0,1)$. We want to show that
\begin{align}
\mathcal{H}_N\left(\rho \XN + (1-\rho)\YN\right) \leq \rho \mathcal{H}_N(\XN) + (1-\rho)\mathcal{H}_N(\YN) -\frac{\ka\rho(1-\rho)}{2}|\YN-\XN|^2.
\end{align}
If $x_i=x_j$ or $y_i=y_j$ for some $1\leq i<j\leq N$, then the right-hand side is infinite and the inequality holds trivially; so, suppose otherwise. Since $V$ is $\ka$-convex, we have for each $i$,
\begin{align}
V(\rho x_i + (1-\rho)y_i) \leq \rho V(x_i) + (1-\rho)V(y_i) - \frac{\ka\rho(1-\rho)}{2}|y_i-x_i|^2.
\end{align}
So, it only remains to show that for each pair $i<j$,
\begin{align}\label{eq:gconv}
\g\left(\left[\rho x_j  + (1-\rho)y_j\right] - \left[\rho x_i + (1-\rho)y_i\right]\right) &= \g\left(\rho(x_j-x_i) + (1-\rho)(y_j-y_i) \right) \nn\\
&\leq \rho\g(x_j-x_i) + (1-\rho)\g(y_j-y_i).
\end{align}
Fix a pair $i<j$. If $x_j-x_i = y_j-y_i$, then there is nothing further to show; so, suppose otherwise. Without loss of generality, suppose $y_j-y_i > x_j-x_i>0$. Then by the fact that
\begin{align}
\forall x>0, \qquad \g''(x) = \begin{cases} \displaystyle \frac{1}{x^2}, & {s=0} \\ \displaystyle \frac{s(s+1)}{|x|^{s+2}}, & {s\neq 0}, \end{cases} 
\end{align}
and therefore $\g$ is convex on $\R_+$, we see that \eqref{eq:gconv} holds.

We perform a qualitative regularization argument that reduces us to the case when $\mathbb{P}_{N,\beta}$ has full support $\R^N$ and $\mathcal{H}_N \in C^2(\R^N)$. Let $\mathbb{G}_N$ be the Gaussian measure with covariance $(\beta\ka)^{-1/2}I_{N\times N}$,
\begin{align}
d\mathbb{G}_N = (2\pi/\beta\ka )^{-\frac{N}{2}}e^{-\frac{\beta\ka|\XN|^2}{2}}d\XN.
\end{align}
Since
\begin{align}
\log\left(\frac{d\mathbb{P}_{N,\be}}{d\mathbb{G}_N}\right) = -\beta\mathcal{H}_N - \log(Z_{N,\be}) + \frac{N}{2}\log\left(\frac{2\pi}{\beta\kappa}\right) + \frac{\beta\kappa|\XN|^2}{2},
\end{align}
we see that $\mathcal{H}_N$ is $\ka$-convex if and only if $\log\left(\frac{d\mathbb{P}_{N,\be}}{d\mathbb{G}_N}\right)$ is concave. Let $\{Q_t\}_{t\geq 0}$ be the Ornstein-Uhlenbeck semigroup with stationary measure $\mathbb{G}_{N}$: for any test function $f$,
\begin{align}
\forall \XN\in \R^N, \qquad (Q_t f)(\XN) \coloneqq \int_{\R^N} f\left(e^{-t}\XN+\sqrt{1-e^{-2t}}\YN \right)d\mathbb{G}_N(\YN).
\end{align}
The measure $\mathbb{G}_N$ is reversible for $\{Q_t\}_{t\geq 0}$. Therefore, $Q_t\#\mathbb{P}_{N,\be}$ is absolutely continuous with respect to $\mathbb{G}_N$, and its Radon-Nikodym derivative $\frac{dQ_t\#\mathbb{P}_{N,\be}}{d\mathbb{G}_N} = Q_t\left(\frac{d\mathbb{P}_{N,\be}}{d\mathbb{G}_N}\right)$. Moreover, as consequence of the Pr\'{e}kopa-Leindler inequality, $Q_t$ preserves log concavity. Hence, $\mathcal{H}_N^t \coloneqq -\frac{1}{\beta}\log\left(Q_t\#\mathbb{P}_{N,\be}\right)$ is $\ka$-convex and belongs to $C^\infty(\R^N)$. Finally, since $\lim_{t\rightarrow 0}(Q_t f)(x) = f(x)$ for any continuous $f$, it follows that $Q_t\#\mathbb{P}_{N,\be}\xrightharpoonup[]{} \mathbb{P}_{N,\be}$ as $t\rightarrow 0$. Thus, if $Q_t\#\mathbb{P}_{N,\be}$ has LSI constant $C_{LS}$ for every $t>0$, then so does $\mathbb{P}_{N,\be}$. 

We proceed under the $C^2$ and full support assumptions. According to Caffarelli's contraction theorem \cite{Caffarelli2000,Caffarelli2000err} (see also \cite{FGP2020} for an alternative proof), if $\mathcal{H}_N$ is $\beta\ka$-convex, then the Brenier map \cite{Brenier1991} $T$ from $\mathbb{G}_N$ to $\mathbb{P}_{N,\be}$ (i.e., $T\#\mathbb{G}_N = \mathbb{P}_{N,\be} $) is $1$-Lipschitz. So, for any test function $\varphi\geq 0$,
\begin{align}
\int_{\R^N}\varphi\log(\varphi) d\mathbb{P}_{N,\be} &= \int_{\R^N}\varphi\log(\varphi) d\left(T\#\mathbb{G}_N\right)  \nn\\
&=\int_{\R^N}(\varphi\circ T) \log\left(\varphi\circ T\right)d\mathbb{G}_N \nn\\
&\leq \frac{2}{\beta\ka}\int_{\R^N} |\nabla(\varphi\circ T)|^2 d\mathbb{G}_N \nn\\
&\leq \frac{2}{\beta\ka}\int_{\R^N}  |(\nabla\varphi)\circ T|^2 |\nabla T|^2 d\mathbb{G}_N \nn\\
&\leq \frac{2}{\beta\ka}\int_{\R^N} |\nabla\varphi|^2 d\mathbb{P}_{N,\be}.
\end{align}
In the third line, we have used the well-known LSI for $\mathbb{G}_N$ \cite{Gross1975}; and in the final line we have used that $\|\nabla T\|_{L^\infty}\leq 1$ together with another application of $T\#\mathbb{G}_N = \mathbb{P}_{N,\be}$. This completes the proof.
\end{proof}

\subsection{Uniform LSI for $\mathbb{Q}_{N,\be}(\mu)$}\label{ssec:LSIQ1D}
Given a density $\mu$, recall from \eqref{QP} and \eqref{defVmub} that
\begin{align}
\mathbb{Q}_{N,\beta}(\mu)= \mathbb{P}_{N,\beta}^{V_{\mu, \beta}}, \qquad \text{where} \quad V_{\mu,\beta}\coloneqq - \g * \mu - \frac1{\beta} \log \mu.
\end{align}
We recycle the notation $\mathcal{H}_N$, so that
\begin{align}
\mathcal{H}_N(\XN) = \sum_{i=1}^N V_{\mu,\beta}(x_i) +  \frac1{2N} \sum_{1\leq i\neq j\leq N} \g(x_i-x_j).
\end{align}
The advantage of this notation is that assuming $V_{\mu,\beta}$ is $\ka$-convex, for some $\ka>0$, we may apply \cref{prop:1DRieszLSI} with $V$ replaced by ${V}_{\mu,\beta}$ to obtain a uniform LSI for $\mathbb{Q}_{N,\beta}(\mu)$.

\begin{prop}\label{prop:1DRieszmLSI1}
Suppose that $\mu\in \P(\R) \cap L^\infty(\R)$ and if $s=0$, also suppose that $\int\log(1+|x|)d\mu<\infty$.\footnote{The $L^\infty$ and log moment assumptions are just to ensure that the convolution $\g\ast\mu$ is well-defined.} For $\beta>0$, suppose that $V_{\mu,\beta}$ is $\ka$-convex, for some $\ka>0$. Then the probability measure $\mathbb{Q}_{N,\beta}(\mu)$ has LSI constant $\frac{2}{\beta\ka}$.
\end{prop}

To give meaning to \cref{prop:1DRieszmLSI1}, we now specify conditions under which $V_{\mu,\beta}$ is uniformly convex.

\begin{lemma}
Let $\mu  \in \P(\R)$ be such that $\log\frac{\mu}{\mu_\be} \in C^2(\R)$.\footnote{Since $\mu_\be \in C^2$, this assumption implies by the chain rule that $\mu\in C^2$.} Suppose $V\in C^2$ and $\be>0$. Then $V_{\mu,\beta}$ is $\ka$-convex with
\begin{align}\label{eq:kaVmu}
\ka \coloneqq \inf V'' -\left(\frac{1}{\beta}\|\log\frac{\mu}{\mu_\be}\|_{\dot{C}^2} + \|\g\ast(\mu-\mu_\be)\|_{\dot{C}^2}\right).
\end{align}
\end{lemma}
\begin{proof}
Recalling the definition of $V_{\mu,\beta}$,
\begin{align}
V_{\mu,\beta} &= - \g\ast(\mu-\mu_\beta + \mu_\beta) - \frac1{\beta} \log(\frac{\mu}{\mu_\beta}\mu_\beta) \nn\\
&=-\g\ast\left(\mu-\mu_\beta\right)-\frac{1}{\beta}\log\frac{\mu}{\mu_\beta} - \left(\g\ast\mu_\beta + \frac{1}{\beta}\log\mu_\beta\right) \nn\\
&=-\g\ast\left(\mu-\mu_\beta\right)-\frac{1}{\beta}\log\frac{\mu}{\mu_\beta} + V - c_\beta,
\end{align}
where to obtain the third line, we have applied \eqref{charactem} to the last term of the second line. By triangle inequality,
\begin{align}
V_{\mu,\beta}'' \geq V'' - \|\g\ast(\mu-\mu_\beta)\|_{\dot{C}^2} -\frac{1}{\beta}\|\log\frac{\mu}{\mu_\beta}\|_{\dot{C}^2},
\end{align}
from which the desired conclusion is immediate.
\end{proof}

\begin{remark}
One can produce probability measures $\mu$ such that $\log\frac{\mu}{\mu_\be} \in C^2$ by choosing $h\in C^2$ and then setting $\mu \coloneqq \frac{e^{h}\mu_\be}{\int e^hd\mu_\be}$, which is tautologically a probability density. One can make the quantities $\|\log\frac{\mu}{\mu_\be}\|_{\dot{C}^2}$, $\|\g\ast(\mu-\mu_\be)\|_{\dot{C}^2}$ arbitrarily small by taking $\|e^{h}-1\|_{C^2}$ arbitrarily small. In particular, we see that there exist non-equilibrium densities $\mu$ such that $V_{\mu,\beta}$ is uniformly convex. 
\end{remark}

 As a corollary in this one-dimensional Riesz case with uniformly convex confinement, suppose we start the dynamics \eqref{eq:SDE} from an initial data $\mu^0$ that is close enough to $\mu_\beta$ in the sense that \eqref{eq:kaVmu} with $\mu = \mu^0$ is strictly positive. If this closeness persists throughout the dynamics \eqref{meanfield} in the sense that \eqref{eq:kaVmu} with $\mu=\mu^t$ is bounded from below by some $\ka_0>0$ uniformly in $t$ (this is a consequence of the aforementioned forthcoming work \cite{HRS2022}), then \cref{thm:main} applies, showing entropic generation of chaos.

\appendix 
\section{Proof of the smallness of free energy in Riesz and regular cases}
In this appendix, we prove the smallness of the free energy \eqref{condK} in the cases \eqref{rieszcase} and in the case of bounded continuous nonnegative interactions, by showing 
\begin{equation}\label{boundcondK}
|\log K_{N,\beta}(\mu)|\le \beta \, o(N), 
\end{equation}
for $o(N)$ independent of $\beta$.

The upper bound is obtained straightforwardly  in the Riesz cases from inserting \eqref{eq:MElbRiesz} into \eqref{defKNbe}, then using that $\mu$ is a probability measure: 
\begin{equation}
\log K_{N,\beta}(\mu) \le  \beta  \begin{cases} \frac{\log(N\|\mu\|_{L^\infty})}{2 \d}\indic_{s=0} + \Cs\|\mu\|_{L^\infty}^{\frac{s}{\d}}N^{\frac{s}{\d}}, & {\d-2\le s <\d} \\ \mathsf{C}\log(N\|\mu\|_{L^\infty})\indic_{s=0} + \mathsf{C}\|\mu\|_{L^\infty}^{\frac{s}{\d}} N^{1 -\frac{2(\d-s)}{2(\d-s)+s(\d+2)}}, &{s<\d-2}.\end{cases}
\end{equation}
In all cases, the preceding right-hand side is $\beta \, o(N)$. When $\g $ is nonnegative and continuous, one can insert the diagonal back into the definition of $F_N$, which implies that 
\begin{align}
F_N(\XN, \mu) \ge - \frac{1}{2N} \g(0,0),
\end{align}
and the proof of the upper bound is concluded in the same way.

The lower bound follows from Jensen's inequality. Indeed,
\begin{equation}\label{lbk}
\log K_{N,\beta}(\mu)  \ge - \beta  N  \mathbb{E}_{\mu^{\otimes N}}\left[F_N(\XN, \mu)\right].
\end{equation}
We then expand out  the definition \eqref{defF} of $F_N$ and use the symmetry of $\g$ to find that 
\begin{align}
\mathbb{E}_{\mu^{\otimes N}}\left[F_N(\XN, \mu)\right] &= \int_{(\R^\d)^N}\Bigg( \frac{1}{2N^2}  \sum_{i \neq j} \g(x_i,x_j) - \frac{1}{N}  \sum_{i=1}^N\int_{\R^\d}  \g(x_i,y)d\mu(y)  \nn\\
&\ph\qquad +\hal\int_{(\R^\d)^2}\g(x,y)d\mu^{\otimes 2}(x,y)\Bigg) d\mu^{\otimes N} (\XN)\nn\\
&=- \frac1{2N}   \int_{(\R^\d)^2} \g(x,y) d\mu^{\otimes 2}(x,y).
\end{align}
Inserting the last line back into the right-hand side of \eqref{lbk} yields 
\begin{align}
\log K_{N,\beta}(\mu)  \ge  \frac{\beta}{2} \int_{(\R^\d)^2}\g(x-y) d\mu^{\otimes2}(x,y),
\end{align}
which gives the desired lower bound in all cases \eqref{rieszcase} and all cases where $\g$ is bounded.


\bibliographystyle{alpha}
\bibliography{../../MASTER}

\newcommand{\etalchar}[1]{$^{#1}$}
\begin{thebibliography}{GOVW09}

\bibitem[ABC{\etalchar{+}}00]{ABCFGIMRS2000}
C\'{e}cile An\'{e}, S\'{e}bastien Blach\`ere, Djalil Chafa\"{\i}, Pierre
  Foug\`eres, Ivan Gentil, Florent Malrieu, Cyril Roberto, and Gr\'{e}gory
  Scheffer.
\newblock {\em Sur les in\'{e}galit\'{e}s de {S}obolev logarithmiques},
  volume~10 of {\em Panoramas et Synth\`eses [Panoramas and Syntheses]}.
\newblock Soci\'{e}t\'{e} Math\'{e}matique de France, Paris, 2000.
\newblock With a preface by Dominique Bakry and Michel Ledoux.

\bibitem[AS21]{AS2021}
Scott Armstrong and Sylvia Serfaty.
\newblock Local laws and rigidity for {C}oulomb gases at any temperature.
\newblock {\em Ann. Probab.}, 49(1):46--121, 2021.

\bibitem[AS22]{AS2022}
Scott Armstrong and Sylvia Serfaty.
\newblock Thermal approximation of the equilibrium measure and obstacle
  problem.
\newblock {\em Ann. Fac. Sci. Toulouse Math. (6)}, 31(4):1085--1110, 2022.

\bibitem[BB19]{BB2019spin}
Roland Bauerschmidt and Thierry Bodineau.
\newblock A very simple proof of the {LSI} for high temperature spin systems.
\newblock {\em J. Funct. Anal.}, 276(8):2582--2588, 2019.

\bibitem[BB21]{BB2021sg}
Roland Bauerschmidt and Thierry Bodineau.
\newblock Log-{S}obolev inequality for the continuum sine-{G}ordon model.
\newblock {\em Comm. Pure Appl. Math.}, 74(10):2064--2113, 2021.

\bibitem[BE85]{BE1985}
D.~Bakry and Michel \'{E}mery.
\newblock Diffusions hypercontractives.
\newblock In {\em S\'{e}minaire de probabilit\'{e}s, {XIX}, 1983/84}, volume
  1123 of {\em Lecture Notes in Math.}, pages 177--206. Springer, Berlin, 1985.

\bibitem[BGL14]{BGL2014}
Dominique Bakry, Ivan Gentil, and Michel Ledoux.
\newblock {\em Analysis and geometry of {M}arkov diffusion operators}, volume
  348 of {\em Grundlehren der mathematischen Wissenschaften [Fundamental
  Principles of Mathematical Sciences]}.
\newblock Springer, Cham, 2014.

\bibitem[BJW19a]{BJW2019edp}
Didier Bresch, Pierre-Emmanuel Jabin, and Zhenfu Wang.
\newblock Modulated free energy and mean field limit.
\newblock {\em S{\'e}minaire Laurent Schwartz--EDP et applications}, pages
  1--22, 2019.

\bibitem[BJW19b]{BJW2019crm}
Didier Bresch, Pierre-Emmanuel Jabin, and Zhenfu Wang.
\newblock On mean-field limits and quantitative estimates with a large class of
  singular kernels: application to the {P}atlak-{K}eller-{S}egel model.
\newblock {\em C. R. Math. Acad. Sci. Paris}, 357(9):708--720, 2019.

\bibitem[BJW20]{BJW2020}
Didier Bresch, Pierre-Emmanuel Jabin, and Zhenfu Wang.
\newblock Mean-field limit and quantitative estimates with singular attractive
  kernels.
\newblock {\em arXiv preprint arXiv:2011.08022}, 2020.

\bibitem[Bre91]{Brenier1991}
Yann Brenier.
\newblock Polar factorization and monotone rearrangement of vector-valued
  functions.
\newblock {\em Comm. Pure Appl. Math.}, 44(4):375--417, 1991.

\bibitem[Bre00]{Brenier2000}
Y.~Brenier.
\newblock Convergence of the {V}lasov-{P}oisson system to the incompressible
  {E}uler equations.
\newblock {\em Comm. Partial Differential Equations}, 25(3-4):737--754, 2000.

\bibitem[Caf00]{Caffarelli2000}
Luis~A. Caffarelli.
\newblock Monotonicity properties of optimal transportation and the {FKG} and
  related inequalities.
\newblock {\em Comm. Math. Phys.}, 214(3):547--563, 2000.

\bibitem[Caf02]{Caffarelli2000err}
Luis~A. Caffarelli.
\newblock Erratum: ``{M}onotonicity of optimal transportation and the {FKG} and
  related inequalities'' [{C}omm. {M}ath. {P}hys. {\bf 214} (2000), no. 3,
  547--563; {MR}1800860 (2002c:60029)].
\newblock {\em Comm. Math. Phys.}, 225(2):449--450, 2002.

\bibitem[CB18]{CB2018}
{L\'{e}na\"{\i}c} Chizat and Francis Bach.
\newblock On the global convergence of gradient descent for over-parameterized
  models using optimal transport.
\newblock In {\em Proceedings of the 32nd International Conference on Neural
  Information Processing Systems}, NIPS18, pages 3040--3050, Red Hook, NY, USA,
  2018. Curran Associates Inc.

\bibitem[CD21]{CD2021}
Louis-Pierre Chaintron and Antoine Diez.
\newblock Propagation of chaos: a review of models, methods and applications.
\newblock {\em arXiv preprint arXiv:2106.14812}, 2021.

\bibitem[CL20]{CL2020}
Djalil Chafa\"{\i} and Joseph Lehec.
\newblock On {P}oincar\'{e} and logarithmic {S}obolev inequalities for a class
  of singular {G}ibbs measures.
\newblock In {\em Geometric aspects of functional analysis. {V}ol. {I}}, volume
  2256 of {\em Lecture Notes in Math.}, pages 219--246. Springer, Cham, [2020]
  \copyright 2020.

\bibitem[dCRS]{CdCRS2023a}
Antonin~Chodron de~Courcel, Matthew Rosenzweig, and Sylvia Serfaty.
\newblock The attractive log gas: phase transitions, (non)uniqueness and
  (in)stability of equilibria, and uniform-in-time propagation of chaos.
\newblock In preparation.

\bibitem[dCRS23]{CdCRS2023}
Antonin~Chodron de~Courcel, Matthew Rosenzweig, and Sylvia Serfaty.
\newblock Sharp uniform-in-time mean-field convergence for singular periodic
  {Riesz} flows.
\newblock {\em arXiv preprint arXiv:2304.05315}, 2023.

\bibitem[DGPS23]{DGPS2023}
Mat\'{i}as~G. Delgadino, Rishabh~S. Gvalani, Grigorios~A. Pavliotis, and
  Scott~A. Smith.
\newblock Phase transitions, logarithmic {S}obolev inequalities, and
  uniform-in-time propagation of chaos for weakly interacting diffusions.
\newblock {\em Communications in Mathematical Physics}, 2023.

\bibitem[Due16]{Duerinckx2016}
M.~Duerinckx.
\newblock Mean-field limits for some {Riesz} interaction gradient flows.
\newblock {\em SIAM Journal on Mathematical Analysis}, 48(3):2269--2300, 2016.

\bibitem[FGP20]{FGP2020}
Max Fathi, Nathael Gozlan, and Maxime Prod'homme.
\newblock A proof of the {C}affarelli contraction theorem via entropic
  regularization.
\newblock {\em Calc. Var. Partial Differential Equations}, 59(3):Paper No. 96,
  18, 2020.

\bibitem[GBM21]{GlBM2021}
Arnaud Guillin, Pierre~Le Bris, and Pierre Monmarch\'{e}.
\newblock Uniform in time propagation of chaos for the 2d vortex model and
  other singular stochastic systems.
\newblock {\em arXiv preprint arXiv:2108.08675}, 2021.

\bibitem[GLBM23]{GlBM2023}
Arnaud Guillin, Pierre Le~Bris, and Pierre Monmarch\'{e}.
\newblock On systems of particles in singular repulsive interaction in
  dimension one: log and {R}iesz gas.
\newblock {\em J. \'{E}c. polytech. Math.}, 10:867--916, 2023.

\bibitem[GLWZ22]{GLWZ2022}
Arnaud Guillin, Wei Liu, Liming Wu, and Chaoen Zhang.
\newblock Uniform {P}oincar\'{e} and logarithmic {S}obolev inequalities for
  mean field particle systems.
\newblock {\em Ann. Appl. Probab.}, 32(3):1590--1614, 2022.

\bibitem[GOVW09]{GOVW2009}
Natalie Grunewald, Felix Otto, C\'{e}dric Villani, and Maria~G. Westdickenberg.
\newblock A two-scale approach to logarithmic {S}obolev inequalities and the
  hydrodynamic limit.
\newblock {\em Ann. Inst. Henri Poincar\'{e} Probab. Stat.}, 45(2):302--351,
  2009.

\bibitem[Gro75]{Gross1975}
Leonard Gross.
\newblock Logarithmic {S}obolev inequalities.
\newblock {\em Amer. J. Math.}, 97(4):1061--1083, 1975.

\bibitem[HM14]{HM2014}
Maxime Hauray and St\'{e}phane Mischler.
\newblock On {K}ac's chaos and related problems.
\newblock {\em J. Funct. Anal.}, 266(10):6055--6157, 2014.

\bibitem[HRS]{HRS2022}
Jiaoyang Huang, Matthew Rosenzweig, and Sylvia Serfaty.
\newblock The modulated free energy method on {$\mathbb{R}^d$}.
\newblock In preparation.

\bibitem[JW17]{JW2017_survey}
Pierre-Emmanuel Jabin and Zhenfu Wang.
\newblock {Mean field limit for stochastic particle systems}.
\newblock In {\em Act. Part. {V}ol. 1. {A}dvances theory, Model. Appl.}, Model.
  Simul. Sci. Eng. Technol., pages 379--402. Birkh{\"{a}}user/Springer, Cham,
  2017.

\bibitem[JW18]{JW2018}
Pierre-Emmanuel Jabin and Zhenfu Wang.
\newblock Quantitative estimates of propagation of chaos for stochastic systems
  with {$W^{-1,\infty}$} kernels.
\newblock {\em Invent. Math.}, 214(1):523--591, 2018.

\bibitem[Lac22]{Lacker2022Gibbs}
Daniel Lacker.
\newblock Quantitative approximate independence for continuous mean field
  {G}ibbs measures.
\newblock {\em Electron. J. Probab.}, 27:Paper No. 15, 21, 2022.

\bibitem[Lac23]{Lacker2023}
Daniel Lacker.
\newblock Hierarchies, entropy, and quantitative propagation of chaos for mean
  field diffusions.
\newblock {\em Probab. Math. Phys.}, 4(2):377--432, 2023.

\bibitem[LLF23]{LlF2023}
Daniel Lacker and Luc Le~Flem.
\newblock Sharp uniform-in-time propagation of chaos.
\newblock {\em Probability Theory and Related Fields}, 2023.

\bibitem[LLN20]{LLN2020}
Tau~Shean Lim, Yulong Lu, and James~H. Nolen.
\newblock Quantitative propagation of chaos in a bimolecular chemical
  reaction-diffusion model.
\newblock {\em SIAM J. Math. Anal.}, 52(2):2098--2133, 2020.

\bibitem[LS18]{LS2018}
Thomas Lebl\'{e} and Sylvia Serfaty.
\newblock Fluctuations of two dimensional {C}oulomb gases.
\newblock {\em Geom. Funct. Anal.}, 28(2):443--508, 2018.

\bibitem[Luk23]{Lukkarinen2023}
Jani Lukkarinen.
\newblock Generation and propagation of chaos in the stochastic {Kac} model,
  2023.
\newblock Talk at Rutgers.

\bibitem[MMM19]{MMM2019}
Song Mei, Theodor Misiakiewicz, and Andrea Montanari.
\newblock Mean-field theory of two-layers neural networks: dimension-free
  bounds and kernel limit.
\newblock In Alina Beygelzimer and Daniel Hsu, editors, {\em Proceedings of the
  Thirty-Second Conference on Learning Theory}, volume~99 of {\em Proceedings
  of Machine Learning Research}, pages 2388--2464, Phoenix, USA, 25--28 Jun
  2019. PMLR.

\bibitem[NRS22]{NRS2021}
Quoc-Hung Nguyen, Matthew Rosenzweig, and Sylvia Serfaty.
\newblock Mean-field limits of {R}iesz-type singular flows.
\newblock {\em Ars Inven. Anal.}, pages Paper No. 4, 45, 2022.

\bibitem[OV00]{OV2000}
F.~Otto and C.~Villani.
\newblock Generalization of an inequality by {T}alagrand and links with the
  logarithmic {S}obolev inequality.
\newblock {\em J. Funct. Anal.}, 173(2):361--400, 2000.

\bibitem[PS17]{PS2017}
Mircea Petrache and Sylvia Serfaty.
\newblock Next order asymptotics and renormalized energy for {R}iesz
  interactions.
\newblock {\em J. Inst. Math. Jussieu}, 16(3):501--569, 2017.

\bibitem[Ros22a]{Rosenzweig2022a}
Matthew Rosenzweig.
\newblock The mean-field approximation for higher-dimensional {Coulomb} flows
  in the scaling-critical {$L^\infty$} space.
\newblock {\em Nonlinearity}, 35(6):2722--2766, may 2022.

\bibitem[Ros22b]{Rosenzweig2022}
Matthew Rosenzweig.
\newblock Mean-{F}ield {C}onvergence of {P}oint {V}ortices to the
  {I}ncompressible {E}uler {E}quation with {V}orticity in {$L^\infty$}.
\newblock {\em Arch. Ration. Mech. Anal.}, 243(3):1361--1431, 2022.

\bibitem[Ros23a]{Rosenzweig2021ne}
Matthew Rosenzweig.
\newblock On the rigorous derivation of the incompressible {E}uler equation
  from {N}ewton's second law.
\newblock {\em Lett. Math. Phys.}, 113(1):Paper No. 13, 32, 2023.

\bibitem[Ros23b]{Rosenzweig2023Kura}
Matthew Rosenzweig.
\newblock Remarks on the logarithmic {Sobolev} inequality for the {Kuramoto}
  model, 2023.
\newblock Available on webpage.

\bibitem[RS]{RS2022}
Matthew Rosenzweig and Sylvia Serfaty.
\newblock Sharp estimates for the variations of {Coulomb} and {Riesz} modulated
  energies, applications to supercritical mean-field limits.
\newblock In preparation.

\bibitem[RS16]{RS2016}
Nicolas Rougerie and Sylvia Serfaty.
\newblock Higher-dimensional {C}oulomb gases and renormalized energy
  functionals.
\newblock {\em Comm. Pure Appl. Math.}, 69(3):519--605, 2016.

\bibitem[RS23]{RS2021}
Matthew Rosenzweig and Sylvia Serfaty.
\newblock Global-in-time mean-field convergence for singular {R}iesz-type
  diffusive flows.
\newblock {\em Ann. Appl. Probab.}, 33(2):754--798, 2023.

\bibitem[RVE22]{RvE2022}
G.~M. Rotskoff and E.~Vanden-Eijnden.
\newblock Trainability and accuracy of artificial neural networks: an
  interacting particle system approach.
\newblock {\em Comm. Pure Appl. Math.}, 75(9):1889--1935, 2022.

\bibitem[Ser20]{Serfaty2020}
Sylvia Serfaty.
\newblock Mean field limit for {Coulomb-type} flows.
\newblock {\em Duke Math. J.}, 169(15):2887--2935, 10 2020.
\newblock Appendix with Mitia Duerinckx.

\bibitem[Ser23]{Serfaty2023}
Sylvia Serfaty.
\newblock Gaussian fluctuations and free energy expansion for {C}oulomb gases
  at any temperature.
\newblock {\em Ann. Inst. Henri Poincar\'{e} Probab. Stat.}, 59(2):1074--1142,
  2023.

\bibitem[SS15]{SS2015}
Etienne Sandier and Sylvia Serfaty.
\newblock 2{D} {C}oulomb gases and the renormalized energy.
\newblock {\em Ann. Probab.}, 43(4):2026--2083, 2015.

\bibitem[Vil04]{Villani2004}
C\'{e}dric Villani.
\newblock Trend to equilibrium for dissipative equations, functional
  inequalities and mass transportation.
\newblock In {\em Recent advances in the theory and applications of mass
  transport}, volume 353 of {\em Contemp. Math.}, pages 95--109. Amer. Math.
  Soc., Providence, RI, 2004.

\bibitem[Zeg92]{Zegarlinski1992}
Bogus\l~aw Zegarli\'{n}ski.
\newblock Dobrushin uniqueness theorem and logarithmic {S}obolev inequalities.
\newblock {\em J. Funct. Anal.}, 105(1):77--111, 1992.

\end{thebibliography}
\end{document}